\documentclass[11pt]{article}
\usepackage{amsthm,enumerate}
\renewcommand{\baselinestretch}{1.2}
\newcommand{\dated}{\mbox{} \hfill {\small [{\tt \today}]}} \usepackage{amsmath,amssymb,amsfonts,diagrams}
\newcommand{\cstar}{{C^\ast}}

\newcommand{\dist}{{\operatorname{dist}}}

\newcommand{\lspan}{{\operatorname{span}}}

\newtheorem{lemma}{Lemma}
\newtheorem*{theorem}{Theorem}
\newtheorem*{corollary}{Corollary}
\theoremstyle{remark}
\newtheorem*{remarks}{Remarks}
\newcommand{\Ball}{\mathrm{Ball}}
\newcommand{\iset}{\mathbb{I}}
\title{A new and simple proof of Schauder's theorem}
\author{\textit{Volker Runde}}
\date{}
\begin{document}
\maketitle
Schauder's theorem from 1930 (\cite{Sch}) asserts that a bounded linear operator between Banach spaces is compact if and only if its adjoint is. Schauder's original proof is completely elementary; at its heart is a diagonal argument that is reminiscent of proofs of the Arzel\`a--Ascoli theorem. Indeed, in \cite{Kak}, S.\ Kakutani gave a proof of Schauder's theorem that invokes the Arzel\`a--Ascoli theorem explicitly. This proof has by now become the canonical one, i.e., it is featured in most textbooks on functional analysis (see \cite{Rud}, \cite{Meg}, or \cite{Fab}, for instance). An alternative proof of Schauder's theorem uses the Alao\u{g}lu--Bourbaki theorem (\cite{Con}).
\par 
In this note, we shall provide another proof of Schauder's theorem, which is both short and completely elementary in the sense that it does not depend on anything beyond basic functional analysis, i.e., the Hahn--Banach theorem and some of its consequences; in particular, we avoid the Arzel\`a--Ascoli theorem (and any kind of related diagonal argument).
\par 
Throughout, we write $\Ball(E)$ for the closed unit ball of a Banach space $E$.
\par 
The following observation was made by H.\ Saar (\cite{Saa}).
\par 
Let $E$ and $F$ be Banach spaces, and let $T \!: E \to F$ be compact. Then $T(\Ball(E))$ is totally bounded, so that, for each $\epsilon > 0$, there are $x_1, \ldots, x_n \in \Ball(E)$ such that, for each $x \in \Ball(E)$, there is $j \in \{ 1, \ldots, n \}$ such that $\| Tx - Tx_j \| < \epsilon$. Letting $Y_\epsilon := \lspan \{ T x_1, \ldots, T x_n \}$. It follows that $\| Q_{Y_\epsilon} T \| < \epsilon$, where $Q_{Y_\epsilon} \!: F \to F / Y_{\epsilon}$ is the quotient map.
\par 
Conversely, suppose that $T \!: E \to F$ is bounded and that, for each $\epsilon > 0$, there is a finite-dimensional subspace $Y_\epsilon$ of $F$ such that $\| Q_{Y_\epsilon} T \| < \epsilon$. Let $\epsilon > 0$ be arbitrary. Then there is a finite-dimensional subspace $Y_\frac{\epsilon}{3}$ such that $\| Q_{Y_\frac{\epsilon}{3}} T \| < \frac{\epsilon}{3}$. Set
\[
  K := \left\{ y \in Y_\frac{\epsilon}{3} : \dist(y, T(\Ball(E))) < \frac{\epsilon}{3} \right\}.
\]
Then $K$ is bounded and thus, as $\dim Y_\frac{\epsilon}{3} < \infty$, totally bounded, i.e., there are $y_1, \ldots, y_m \in K$ such that, for each $y \in K$, there is $j \in \{ 1, \ldots, m \}$ such that $\| y - y_j \| < \frac{\epsilon}{3}$. By the definition of $K$, there is, for each $j =1, \ldots, m$, an element $x_j \in \Ball(E)$ with $\| y_j - Tx_j \| < \frac{\epsilon}{3}$. Let $x \in \Ball(E)$ be arbitrary. Since $\| Q_{Y_\frac{\epsilon}{3}} T \| < \frac{\epsilon}{3}$, there is $y \in K$ such that $\| Tx - y \| < \frac{\epsilon}{3}$. Let $j \in \{1, \ldots, m \}$ such that $\| y - y_j \| < \frac{\epsilon}{3}$. It follows that
\[
  \| Tx - Tx_j \| \leq \| T x - y \| + \| y - y_j \| + \| y_j - Tx_j \| < \frac{\epsilon}{3} + \frac{\epsilon}{3} +\frac{\epsilon}{3}
  = \epsilon.
\]
Hence, $T(\Ball(E))$ is totally bounded, and $T$ is compact.
\par
We thus have:
\begin{lemma} \label{l1}
Let $E$ and $F$ be Banach spaces, and let $T \!: E \to F$ be bounded. Then $T$ is compact if and only if, for each $\epsilon > 0$, there is a finite-dimensional subspace $Y_\epsilon$ of $F$ such that $\| Q_{Y_\epsilon} T \| < \epsilon$, where $Q_{Y_\epsilon} \!: F \to F / Y_\epsilon$ is the quotient map.
\end{lemma}
\par 
Our next lemma requires the Hahn--Banach theorem.
\begin{lemma} \label{l2}
Let $E$ and $F$ be Banach spaces, let $T \!: E \to F$ be bounded, let $\epsilon > 0$, and let $X$ be a closed subspace of $E$ with finite codimension such that $\| T |_X \| < \frac{\epsilon}{3}$. Then there is a finite-dimensional subspace $X_0$ of $E$ such that $\| Q_{T X_0} T \| < \epsilon$, where $Q_{TX_0} \!: F \to F / TX_0$ is the quotient map.
\end{lemma}
\begin{proof}
Using the Hahn--Banach theorem, we can embed $F$ isometrically into $\ell^\infty(\iset)$ for a suitable index set ($\iset = \Ball(E^\ast)$ will do). Hence, we can suppose without loss of generality that $F = \ell^\infty(\iset)$.
\par 
Applying the Hahn--Banach theorem coordinatewise, we obtain an operator $\tilde{T} \!: E \to \ell^\infty(\iset)$ such that $\tilde{T} |_X = T |_X$ and $\| \tilde{T} \| = \| T |_X \| < \frac{\epsilon}{3}$. Set $S := T-\tilde{T}$. Then $S$ vanishes on $X$, and since $X$ has finite codimension this means that $S$ is a finite rank operator and thus compact. Consequently, there are $x_1, \ldots, x_n \in \Ball(E)$ such that, for each $x \in \Ball(E)$, there is $j \in \{ 1, \ldots, n \}$ with $\| Sx - Sx_j \| < \frac{\epsilon}{3}$.
\par 
Fix $x \in \Ball(E)$, let $j \in \{ 1, \ldots, n \}$ be such that $\| Sx - Sx_j \| < \frac{\epsilon}{3}$, and note that
\[
  \| Tx - Tx_j \| \leq \| Sx - Sx_j \| + \| \tilde{T} x - \tilde{T} x_j \| 
  < \frac{\epsilon}{3} + 2 \| \tilde{T} \| <
  \frac{\epsilon}{3} + 2 \frac{\epsilon}{3} = \epsilon.
\]
The space $X_0 = \lspan \{ x_1, \ldots, x_n \}$ thus has the desired property.
\end{proof}
\par 
In conjunction, Lemmas \ref{l1} and \ref{l2} yield immediately:
\begin{corollary}
Let $E$ and $F$ be Banach spaces, and let $T \!: E \to F$ be bounded with the following property: for each $\epsilon > 0$, there is a closed subspace $X_\epsilon$ of $E$ with finite codimension such that $\| T |_{X_\epsilon} \| < \epsilon$. Then $T$ is compact.
\end{corollary}
\par 
We can now prove Schauder's theorem:
\begin{theorem}
Let $E$ and $F$ be Banach spaces, and let $T \!: E \to F$ be a bounded linear operator. Then the following are equivalent:
\begin{enumerate}[\rm (i)]
\item $T$ is compact;
\item for each $\epsilon > 0$, there is a finite-dimensional subspace $Y_\epsilon$ of $F$ such that $\| Q_{Y_\epsilon} T \| < \epsilon$, where $Q_{Y_\epsilon} \!: F \to F / Y_\epsilon$ is the quotient map;
\item for each $\epsilon > 0$, there is a closed subspace $X_\epsilon$ of $E$ with finite codimension such that $\| T |_{X_\epsilon} \| < \epsilon$;
\item $T^\ast \!: F^\ast \to E^\ast$ is compact.
\end{enumerate}
\end{theorem}
\begin{proof}
(i) $\Longleftrightarrow$ (ii) is Lemma \ref{l1}, and (iii) $\Longrightarrow$ (ii) follows from Lemma \ref{l2}. 
\par 
(ii) $\Longrightarrow$ (iv): Let $\epsilon > 0$, and let $Y_\epsilon$ be a finite-dimensional subspace of $F$ such that $\| Q_{Y_\epsilon} T \| < \epsilon$. Let $X_\epsilon$ be the annihilator of $Y_\epsilon$ in $F^\ast$, so that $X_\epsilon$ has finite codimension in $F^\ast$, $T^\ast |_{X_\epsilon} = (Q_{Y_\epsilon}T)^\ast$, and thus $\| T^\ast |_{X_\epsilon} \| < \epsilon$. Since $\epsilon > 0$ was arbitrary, the Corollary---applied to $T^\ast$---thus yields (iv).
\par 
(iv) $\Longrightarrow$ (iii): Let $\epsilon > 0$. Invoking Lemma \ref{l1} for $T^\ast$ and then arguing as in the proof of (ii) $\Longrightarrow$ (iv), we obtain a closed subspace $X_\epsilon$ of $E^{\ast\ast}$ with finite codimension such that $\| T^{\ast\ast} |_{X_\epsilon} \| < \epsilon$. Consequently, $\| T |_{X_\epsilon \cap E} \| < \epsilon$ holds as well. Since $\epsilon > 0$ was arbitrary, the Corollary implies (i).
\end{proof}
\begin{remarks} \begin{enumerate}
\item As M.\ Cwikel pointed out to me, the equivalence of (i) and (iii) in the Theorem was already obtained by H.\ E.\ Lacey in \cite{Lac}. His result is reproduced with a proof on \cite[p.\ 91]{Pie}.
\item A similar proof of Schauder's Theorem---in the sense that it is both elementary and avoids diagonal arguments---is given in \cite{Gol} and \cite{Kre}: this was pointed out to me by D.\ Werner and A.\ Valette, respectively.
\end{enumerate} \end{remarks}
\renewcommand{\baselinestretch}{1.0}
\dated
\vfill
\begin{tabbing}
\textit{Author's address}: \= Department of Mathematical and Statistical Sciences \\
\> University of Alberta \\
\> Edmonton, Alberta \\
\> Canada T6G 2G1 \\[\medskipamount]
\textit{E-mail}: \> \texttt{vrunde@ualberta.ca} \\[\medskipamount]
\textit{URL}: \> \texttt{http://www.math.ualberta.ca/$^\sim$runde/}
\end{tabbing}

\end{document}